\documentclass[10pt,a4paper]{article}
\usepackage{geometry}
\usepackage[english]{babel}
\usepackage{amsmath}
\usepackage{graphicx}
\usepackage{ marvosym }
\usepackage[utf8]{inputenc}
\usepackage{mathtools}
\usepackage{geometry}
\usepackage{amsfonts}
\usepackage{amssymb}
\usepackage{amsthm}
\usepackage{t1enc}
\usepackage[titles]{tocloft}
\usepackage{makeidx}
\usepackage{wasysym}
\usepackage{stmaryrd}
\usepackage{calc}  
\usepackage{enumitem} 
\usepackage[refpage]{nomencl}

\usepackage[colorlinks=true,urlcolor=blue, linkcolor=blue,pageanchor=false, backref]{hyperref}
\usepackage{algpseudocode}
\usepackage{tikz}
\usetikzlibrary{arrows}
\usepackage{indentfirst}
\usepackage{xcolor}
\usepackage{ dsfont }
\usepackage{float}
\usepackage[abbrev,msc-links,backrefs]{amsrefs} 
\usepackage{doi}

\renewcommand{\PrintDOI}[1]{\doi{#1}}

\theoremstyle{plain}
\newtheorem{thm}{Theorem}[section]
\newtheorem{ques}[thm]{Question}
\newtheorem{claim}[thm]{Claim}

\newtheorem{prop}[thm]{Proposition}

\newtheorem*{prop*}{Proposition}
\newtheorem*{seged*}{Sublemma}
\newtheorem{cor}[thm]{Corollary}
\newtheorem{lem}[thm]{Lemma}

\newtheorem*{cond*}{Condition}
\newtheorem{conj}[thm]{Conjecture}

\newtheorem*{lem*}{Lemma}
\theoremstyle{definition}

\newtheorem*{defn*}{Definition}

\newtheorem{fel*}[thm]{Exercise}

\newtheorem*{megf*}{Observation}
\theoremstyle{remark}
\newtheorem{rem}[thm]{Remark}

\newtheorem*{rem*}{Remark}

\newenvironment{sbiz}{\par\noindent{\itshape Proof:}\ }{\rule{1ex}{1.5ex}}

\title{Vertex-flames in countable rooted digraphs preserving an Erd\H os-Menger separation for each vertex}
\author{Attila Joó \thanks{Alfréd Rényi Institute of Mathematics,
Hungarian Academy of Sciences; MTA-ELTE Egerv\'ary Research Group. The research was supported by OTKA 113047.
 Email: {\tt jooattila@renyi.hu
 } }}
\date{2019}
\begin{document}
\maketitle
\begin{abstract}
It follows from a theorem of Lov\'asz  that if $ D $ is a finite digraph with $ r\in 
V(D) $ then there is a spanning subdigraph
$ E $ of $ D $ such that for every vertex $ v\neq r $ the following  quantities are equal: the local connectivity from $ r $ 
to $ v $ in $ D $, the local connectivity from 
$ r $ to $ v $ in $ E $ and the indegree of $ v $ in $ E $.

In infinite combinatorics cardinality is often an overly rough measure to obtain deep  results and it is more fruitful 
to capture structural properties instead of just equalities between certain quantities. The best known example for such a result 
is the generalization of Menger's theorem to infinite digraphs. 
We 
generalize the result of Lov\'asz above in this spirit.
Our main result is that every countable $ r $-rooted digraph $ D $ has a spanning subdigraph $ E $ with the following 
property.   For
every $ v\neq r $, $ E $  contains a system $ \mathcal{R}_v $ of internally disjoint $ r\rightarrow v $ paths such that the 
ingoing 
edges of $ v $ in $ E $ are exactly the last edges of the paths in $ \mathcal{R}_v $. Furthermore, the path-system $ 
\mathcal{R}_v $ is 
``big'' in $ D $ in the Erdős-Menger sense, i.e., one can choose from each path in 
$ \mathcal{R}_{v} $ either an edge or an internal vertex in such a way that a resulting set separates $ v $ from $ r $ in $ D $.
\end{abstract}

\section{Introduction}
Small subgraphs witnessing some kind of connectivity property  play an important role in graph theory. Let us recall 
a result of L. Lov\'asz of this manner. Consider a finite digraph $ D $  with a given root $ r\in V(D) $.   We are looking for a 
spanning subdigraph $ E $ of $ D $ that 
preserves the local 
vertex-connectivities from $ r $ (i.e.,   $ \boldsymbol{\kappa_D(r,v)}=\kappa_E(r,v) $ holds 
for every $ v\in V(D)-r $) with  a minimal possible number of edges. For every   $ v\in V(D)-r $, we need to keep at least 
$ \kappa_{D}(r,v) $ ingoing edges 
hence $ \sum_{v\in V(D)-r}\kappa_D(r,v)$ is a trivial lower bound for the number of edges of $ E $. 
Surprisingly this lower bound is always sharp.

\begin{thm}
[L. Lov\'asz]\label{Flame alaptetel}
If $ D $ is a finite digraph and $ r\in V(D) $, then there is a spanning subdigraph $ E $ of $ D $ such that for every $ 
v\in 
V(D)-r $
\[  \kappa_D(r,v)=\kappa_{E}(r,v)= \left|\mathsf{in}_E(v)\right|,\]
where $ \boldsymbol{\mathsf{in}_E(v)} $ is the set of the ingoing edges of $ v $ in $ E $.
\end{thm}

 The main result of this paper is a generalization of Theorem \ref{Flame alaptetel} to countable digraphs.  
The equations in Theorem \ref{Flame alaptetel} make sense in infinite digraphs as well. Even 
so, cardinality is an overly rough measure to give a satisfactory  generalization. Instead of the equation 
 $ \kappa_E(r,v)=\left|\mathsf{in}_E(v)\right| $ we demand the existence of a system $ \mathcal{P} $ of internally 
 disjoint directed paths from $ r $ to $ v $ in $ E $ such 
that the set of the last edges of the paths in $ \mathcal{P} $ is  $\mathsf{in}_E(v) $. If $ E $ satisfies this for each vertex 
$ v $, then  $ E $ is called a \textbf{vertex-flame} with respect to the root $ r $ (the name ``flame'' was given by G. 
Calvillo-Vives  who rediscovered Theorem \ref{Flame alaptetel} independently in his Ph.D. thesis \cite{flameVives}).

The equation $ \kappa_E(r,v)=\kappa_D(r,v) $ means that some maximal-sized internally disjoint $ r\rightarrow v $ 
path-system $ \mathcal{P} $ of $ D $ lies in $ E $. We want $ \mathcal{P} $ to be ``big'' in $ D $ not just cardinality-wise but in the  
Erdős-Menger sense. 
More precisely, we demand that one can choose from every  $P\in  \mathcal{P} $ either one internal vertex or an edge such 
that the resulting set separates $ v $ from $ r $, i.e., meets every $ r\rightarrow v $ path of  $ D $. Note that  the 
separation can be chosen as a vertex set $ S $ whenever $ rv\notin D $ and in the form $ S\cup \{ rv \} $ otherwise. The set  
of 
internally 
disjoint $ r\rightarrow 
v $ path-systems $ \mathcal{P} $ admitting such a separation  is denoted by $\boldsymbol{\mathfrak{I}_D(v)} $. It is easy 
to see that  the 
Aharoni-Berger theorem  
ensures that $ \mathfrak{I}_D(v)\neq \varnothing $  (see Theorem \ref{inf 
Menger} and the paragraph after it).
We define a spanning subdigraph $ L $ of $ D $ to be 
$ \boldsymbol{D} $\textbf{-vertex-large} if  for 
each vertex $ v\neq r $ of $ D $ there is a $ 
\mathcal{P}\in 
\mathfrak{I}_D(v) $ that lies in $ L $.  The main result of this paper is the following theorem. 

\begin{thm}\label{main result flame}
If $ D $ is a countable digraph with a given root vertex $ r $ then there exists a $ D $-vertex-large spanning 
subdigraph $ E $ of $ D $ which is a vertex-flame.
\end{thm}

The paper is organized as follows. In the next subsection we present proof methods that work for finite digraphs but fail or 
insufficient for  
infinite ones and we introduce our proof strategy for countable digraphs. At the end of the first section we introduce some 
further
notation. In the second section we state our key lemmas without proofs and derive the main result from them in a single page. 
The third section is  devoted to the 
proofs of the key lemmas. In the last section we discuss some  open problems.\\

\textbf{Acknowledgement:} The author is very grateful for the thorough work of the referees. The paper improved a lot 
through their suggestions.
 
\begin{rem}\label{vertex from edge}
Lov\'asz proved Theorem \ref{Flame alaptetel} originally for edge-connectivity instead of vertex-connectivity  (see Theorem 2 
of 
\cite{lovasz1973connectivity})  from which the vertex version
follows. Indeed, let $ D' $ be the digraph
that we obtain from $ D $ by splitting every $ v\in V(D)-r $ to an edge $ t_vh_v $ where  $ t_v $ inherits the ingoing and $ 
h_v $ 
the outgoing edges of $ v $. Observe that the systems of edge-disjoint $ r\rightarrow t_v $ paths of $ D' $ and the systems 
of 
internally disjoint $ r\rightarrow v $ paths in $ D $ are in a natural correspondence. Let $ E' $ be that  we obtain by applying 
the edge version of Theorem 
\ref{Flame 
alaptetel} to the digraph $ D' $. We define $ E $ to be the spanning subdigraph of $ D $ 
consists of the common edges of $ D  $
and $ E' $. It is easy to check that $ E $ 
satisfies the conditions of Theorem \ref{Flame alaptetel}.
\end{rem}

\begin{rem}
Seemingly we promised a stronger property for $E  $ in the abstract than in Theorem \ref{main result flame}. Namely an $ 
\mathcal{R}_v $ in $ E $ for $ v\in V(D)-r 
$ 
that 
witnesses 
simultaneously the $ D $-vertex-largeness  and the vertex-flame property at $ v $.
If $ \mathcal{P}_v $ exemplifies the $ D $-vertex-largeness and $ \mathcal{Q}_v $ shows the vertex-flame property at $ v $ 
in $ E $ then an easy application of Pym's theorem (Theorem \ref{Pym}) results in an $ \mathcal{R}_v  $ that witnesses 
both. Hence the property of $ E $ given in the abstract is equivalent with two properties demanded in Theorem \ref{main 
result 
flame}.
\end{rem}
\subsection{Proof strategies informally}
One possible proof strategy, the original approach of Lov\'asz, for Theorem \ref{Flame alaptetel} is ``trimming'' $ D $ while  
keeping 
vertex-largeness. One can show for example 
that if $ \mathcal{P}$ is a system of internally disjoint $ r\rightarrow u $ paths of size $ \kappa_{D}(r,u) $ and we delete those $ 
e\in \mathsf{in}_D(u)  $ that 
are unused by $ \mathcal{P} $, then  $ \kappa_{L}(r,v) =\kappa_{D}(r,v) $ holds for $ v\in V(D)-r $ and of course
$ \kappa_{L}(r,u)=\left|\mathsf{in}_L(u)  \right|$ holds as well. Theorem \ref{Flame alaptetel} follows by applying this for each 
vertex one by one.

Our approach for the finite case was adding new edges repeatedly having a vertex-flame in each step.  We will see that if a 
vertex-flame 
$ 
F $ is not $ D $-vertex-large, then one can properly extend $ F $ with a suitable edge of $ D $ such that the result is still a 
vertex-flame. By iterating this, we can extend any vertex-flame  of $ D $ to  a $ D $-vertex-large vertex-flame whenever $ D 
$ is  
finite 
(actually the assumption ``$ \kappa_D(r,v) <\aleph_0$ for every $ v\in V(D)-r $'' is enough).

It will turn out that the key ideas of the proof sketches above still work in the general case, and moreover, they are 
compatible with 
our 
stronger definitions.  For example if $ \mathcal{P}\in \mathfrak{I}_D(v) $ then the $ L $ that we obtain from $ D $ by the 
deletion of  those 
ingoing edges of $ v $ that are unused by $ \mathcal{P} $
is $ D $-vertex-large  (not just $ \kappa_{L}(r,v) =\kappa_{D}(r,v) $ holds for $ v\in V(D)-r $). The main difficulty is that by 
iterating the deletions or extensions infinitely many times we may lose at a limit step the  property we intended to keep. In the 
case of the 
deletions, the situation is actually worse. In the finite case, vertex-largeness is transitive in the sense that if $ L_1 $ is $ D 
$-vertex-large and $ L_2 $ is $ L_1
$-vertex-large, then $ L_2 $ is $ D $-vertex-large. (Indeed, if $ \kappa_D(r,v)= \kappa_{L_1}(r,v) $ and 
$ \kappa_{L_2}(r,v)= \kappa_{L_1}(r,v) $  for every $ v $, then $ \kappa_D(r,v)= \kappa_{L_1}(r,v) $ for every $ v $.)
Examples show that this transitivity of vertex-largeness does not hold in general, thus applying  twice a vertex-largeness 
preserving  edge-deletion may already be problematic.

Our proof strategy for the countable case is a mixture of the two approaches above. In every step we fix some edges and 
delete 
some others. In a general 
step we fix the edges of a $ \mathcal{P}\in \mathfrak{I}_D(v) $ for the next $ v $ with respect to a fixed enumeration in such a way 
that $ 
\mathcal{P} $ covers all the 
(finitely many) ingoing edges of $ v $ that are already fixed. Right after this we delete all the ingoing edges of $ v $ that are 
unused by $ 
\mathcal{P} $. It turns out that this way we keep $ D $-vertex-largeness, and furthermore, $ \mathcal{P} $ witnesses that in 
the 
final 
digraph we do not violate the vertex-flame property at $ v $. A critical part of the proof is to guarantee that each step we are 
really
able to cover a finite subset of ingoing edges of a given vertex $ v $ by a system of internally disjoint $ r\rightarrow v $ 
paths.  A 
rooted
digraph $ F $ is called a \textbf{quasi-vertex-flame} if for each vertex $ v\in V(D)-r 
$ every 
finite subset of $ \mathsf{in}_F(v) $
can be covered in $ F $ by an internally disjoint $ r\rightarrow v $ path-system.  It turns out that our iterative 
process maintains 
 the quasi-vertex-flame property assuming we have it at the beginning. We show that a maximal quasi-vertex-flame $ F $ in $ 
 D $ has a very strong property (it is ``vertex-largeness faithful'' with respect to $ D $) that allows us to replace $ D $ by $ F 
 $ before starting the process described above.

\subsection{Notation}
We apply some  standard notation from set theory. Variables $ 
\boldsymbol{\alpha,\beta,\gamma} $ stand for ordinals,  the smallest limit ordinal (i.e., the set of the natural numbers) is 
 $ \boldsymbol{\omega} $. For a 
family of sets $ \mathcal{X} $, the 
union of 
the elements of $ \mathcal{X} $ is denoted by $ \boldsymbol{\bigcup \mathcal{X}} $. 
For an ordered pair $ \{ \{ u \}, \{ u,v \} \} $, we write simply  $ \boldsymbol{uv} $. We use the abbreviations $ 
\boldsymbol{X-x} $ 
and $ \boldsymbol{X+x }$ 
for $ X\setminus \{ x \} $ and $ X\cup \{ x \} $ respectively.   

Let an infinite vertex set $ \boldsymbol{V} $ and a ``root vertex'' $ \textbf{r}\in V $  be fixed through the paper. A \textbf{digraph} is a subset 
of $ 
V\times V $. The vertex set $ V $ and hence the digraphs may have arbitrary large infinite size (except in the proof of the 
main result in section 2 where we will
restrict it to countably infinite).  The 
set of the ingoing and outgoing
 edges of a $ v\in V $ with 
 respect to $ D $ is
 denoted by $ \boldsymbol{\mathsf{in}_D(v)}$ and $ \boldsymbol{\mathsf{out}_D(v)} $ respectively.  For 
 the set of the in-neighbours of a vertex $ v $ we write  $ \boldsymbol{N^{\text{\textbf{in}}}_D(v)}$, and $ 
 \boldsymbol{N^{\text{\textbf{out}}}_D(v)} 
 $ stands for the out-neighbours. In several definitions and statements the root $ r $ will play a special role while the ingoing 
 edges of 
  $ r $ are irrelevant. This motivates to define a  \textbf{rooted digraph} as a digraph $ D $ with $ 
  \mathsf{in}_D(r)=\varnothing 
  $.

 A $ v_0\rightarrow v_n $ path for $ v_0\neq v_n\in V $ is a digraph $ P=\{ v_0v_1,v_1v_2,\dots,v_{n-1}v_n \} $ where $ v_i\in 
 V $ 
 are pairwise distinct.  The singleton $ \{ v \} $ is
 considered a  $ v\rightarrow v $ path. We say that $ P $ is an 
   $ \boldsymbol{X \rightarrow Y} $ 
   \textbf{path} for some $ X,Y\subseteq V $ if exactly the first vertex of $ P $ is in $ X $ and exactly the last is in $ Y $. Let $ D 
   $ be a digraph and let $ 
   X,Y,S\subseteq V $.  If every $ 
     X\rightarrow Y $ path of $ D $ meets $ S\subseteq V $, then  we say that $ S $ \textbf{separates} $ Y $ from $ X $ (or $ S $ is 
     an $ \boldsymbol{XY} $\textbf{-separation}) in $ D $.  A system $ 
  \mathcal{P} $ 
  of $ u\rightarrow v $ paths is \textbf{internally disjoint} if the common vertices of any  $P\neq Q\in 
  \mathcal{P} $ 
  are exactly $ u $ and 
  $ v $.  A path-system $ \mathcal{P} $ is an $ \boldsymbol{r} $\textbf{-fan} if  any two 
  distinct paths in   $\mathcal{P} $ have only their initial vertex $ r $ in common. If the terminal vertex $ v $ is the only 
  common vertex
  then we call the path-system a $ \boldsymbol{v} $\textbf{-infan}. 
Let $ \boldsymbol{V_{\text{\textbf{first}}}(\mathcal{P})}
 $ be the set of the first vertices of the paths in $ \mathcal{P} $ and we define $ 
 \boldsymbol{V_{\text{\textbf{last}}}(\mathcal{P})}$  analogously.
 The set of the last edges of the paths in $ \mathcal{P} $ is denoted  by  $ \boldsymbol{A_{\text{\textbf{last}}}( 
 \mathcal{P})}$. 
 For a rooted digraph $ D $ and $ v\in V-r $, let us denote by $ \boldsymbol{\mathcal{G}_{D}(v)} $ the set of 
 those $ I\subseteq 
 \mathsf{in}_D(v) $ for which there is a 
 system $ 
 \mathcal{P} $ of internally disjoint $ r \rightarrow v $ paths in $ D $ with $ A_{\text{last}}(\mathcal{P})=I $.   The rooted 
 digraph $F$ is 
 a \textbf{vertex-flame} if  $ \mathsf{in}_F(v)\in \mathcal{G}_{F}(v) $ for every $ v\in V-r $. 
 If  for 
 every $ v\in V-r$ and for 
 every \emph{finite}  $ I\subseteq \mathsf{in}_{F}(v) $  we have $ I\in \mathcal{G}_{F}(v) $, then  $ F $ is defined to be a 
 \textbf{quasi-vertex-flame}. To 
 improve the flow 
 of words, we write simply flame, quasi-flame and large 
 instead of 
 vertex-flame, quasi-vertex-flame and vertex-large (except in Theorems, Lemmas etc.). The edge version of these concepts 
 appear only 
 among the open 
 problems in the last section hence it will not lead to confusion.

\section{The proof of the main result}
In this section we state our key lemmas without proofs  and derive our main result from them.

\begin{lem}\label{large quasi flame}
For every rooted digraph $ D $, there is a quasi-vertex-flame $ F\subseteq D $ such that whenever an $ L\subseteq F $ is $ F 
$-vertex-large it is $ D $-vertex-large as well. 
\end{lem}

\begin{lem}\label{key lemma}
Let $ L\subseteq D $ be rooted digraphs such that for every $ v\in V-r $ with  $ \mathsf{in}_L(v)\subsetneq \mathsf{in}_D(v) $ 
there is a $ \mathcal{P}\in \mathfrak{I}_D(v) $ that lies in $ L $. Then $ L $ is 
$ D $-vertex-large.
\end{lem}

\begin{lem}\label{key lemma1}
If   $ D $ is a quasi-vertex-flame and $ L $ is $ D $-vertex-large then $ L $ is a quasi-vertex-flame as well.
\end{lem}

\begin{thm}[Pym, \cite{pym1969linking}]\label{Pym}
Let $ D $ be a digraph and let $ \mathcal{P},\mathcal{Q} $ be systems of disjoint $ X\rightarrow Y $ paths for some 
$ X,Y\subseteq V$. Then there 
is a system $ \mathcal{R} $ of disjoint $ X\rightarrow Y $ paths for which $ V_{\text{first}}(\mathcal{R}) \supseteq 
V_{\text{first}}(\mathcal{P}) 
$ and $ V_{\text{last}}(\mathcal{R}) \supseteq V_{\text{last}}(\mathcal{Q}) $. 
\end{thm}
\begin{proof}[Proof of Theorem \ref{main result flame}:]
Let $V=\{ v_n 
\}_{n<\omega} $. We may assume by Lemma \ref{large quasi flame}  that $ D $ 
is a 
quasi-flame.

We construct by recursion a sequence  $ (\mathcal{P}_n)_{n<\omega} $   such that for every $ n<\omega $: 

\begin{enumerate}

\item $ \mathcal{P}_n\in \mathfrak{I}_{D}(v_n)$,
\item $ A_{\text{last}}(\mathcal{P}_n) \supseteq \bigcup_{m<n}\mathsf{in}_{\mathcal{P}_m}(v_n) $,
\item $ \mathsf{in}_{\mathcal{P}_{n}}(v_m)\subseteq A_{\text{last}}(\mathcal{P}_m) $ for $ m<n $.
\end{enumerate}

Let us show first that if the construction is done, then the union $ E $ of the edge sets of the path-systems $ \mathcal{P}_n $ 
form a $ D $-large flame. Indeed, largeness of $ E $ follows immediately from property 1. Properties 2 and 3 state 
together  that 
$ \mathsf{in}_{\mathcal{P}_m}(v_n)\subseteq A_{\text{last}}(\mathcal{P}_n) $ for $ m,n<\omega $  thus  
 $ \mathcal{P}_n $ ensures $ \mathsf{in}_{E}(v_n)\in \mathcal{G}_{E}(v_n) $ which means that $ E $ 
 is a flame. 

Let $ \mathcal{P}_0\in \mathfrak{I}_{D}(v_0)$ be arbitrary. Suppose 
that $ 
\mathcal{P}_m $ is defined for 
$ m<n $ where $ n>0 $ and so 
far the conditions hold. Delete those ingoing edges of $ v_0, v_1,\dots ,v_{n-1} $ from $ D $ that we cannot use in the 
construction of 
$ \mathcal{P}_n $ according to property 3 and let us denote  the remaining digraph by $ D_n $. Since properties 2 and 3 hold 
so far,  for $ \ell, m<n $ we have  $\mathsf{in}_{\mathcal{P}_\ell}(v_m)\subseteq A_{\text{last}}(\mathcal{P}_m) $. 
Therefore 
 $ D_n $ 
contains the path-systems $ 
\mathcal{P}_0,\dots, 
\mathcal{P}_{n-1} $. Thus we may conclude by Lemma \ref{key lemma}  that $ 
D_n $ is  $ D $-large. Hence Lemma  \ref{key lemma1} guarantees that $ D_n $ is a 
quasi-flame. Take a $ \mathcal{P}\in \mathfrak{I}_D(v_n) $ that lies in $ D_n $ and take an $ S $ consisting of choosing 
exactly one internal vertex from each path in $ \mathcal{P}-\{ rv_n \} $ that separates 
$ v_n $ from $ r $  in $ D-rv_n $.  Let $ 
\mathcal{P}' $ 
consist of 
the segments of paths in $ \mathcal{P}-\{ rv_n \} $ from $ S $ to $ N^{\mathsf{in}}_{D-rv}(v_n) $ (see Figure \ref{fig 
mainthm}).  Let us denote 
$ (\bigcup_{m<n}\mathsf{in}_{\mathcal{P}_m}(v_n))-rv_n $ by $ J $. Note that 
$ \left|J\right|\leq n $ since each of $ \mathcal{P}_0,\dots ,\mathcal{P}_{n-1} $ uses at most one ingoing edge of $ v_n 
$. Then $ J\in 
\mathcal{G}_{D_n}(v_n) $ because $ D_n $ is a quasi-flame.  Take a $ \mathcal{Q} $ that 
witnesses $ J\in 
\mathcal{G}_{D_n}(v_n) $ and 
let $ \mathcal{Q}' $ be the set of the segments of the paths in $ \mathcal{Q} $ from the last intersection with 
$ S $ to $ N^{\text{in}}_{D-rv_n}(v_n) $. 

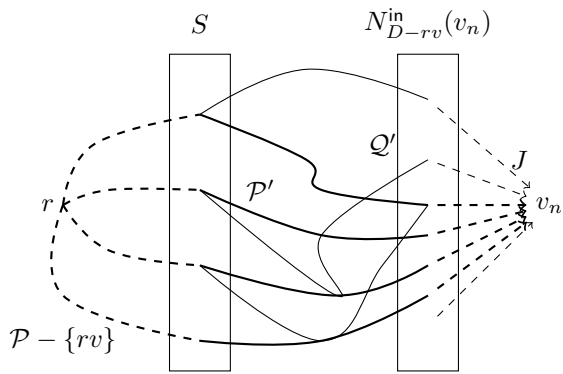
\begin{figure}[H]
\centering
\begin{tikzpicture}

\draw  (-0.6,0.2) rectangle (0.2,-4);
\draw  (2.4,0.2) rectangle (3.2,-4);

\node at (-0.2,0.6) {$S$};

\node at (2.8,0.6) {$N^{\mathsf{in}}_{D-rv}(v_n)$};
\node (v7) at (-2,-1.8) {};
\node (v7') at (-2.2,-1.8) {$ r $};
\node (v2) at (4.4,-1.8) {$v_n$};
\node at (4,-1.2) {$J$};

\node (v1) at (2.8,-0.4) {};
\node (v3) at (2.8,-1.2) {};
\node (v4) at (2.8,-1.8) {};
\node (v5) at (2.8,-2.6) {};
\node (v6) at (2.8,-3.4) {};
\node (v13) at (-0.2,-3.6) {};
\node (v12) at (-0.2,-2.6) {};

\node (v11) at (-0.2,-1.6) {};
\node (v10) at (-0.2,-0.6) {};
\node (v8) at (2.8,-2.2) {};
\node (v9) at (2.8,-3) {};

\draw[->, dashed]   (v1) edge (v2);
\draw[->, dashed]   (v3) edge (v2);
\draw[thick,->, dashed]   (v4) edge (v2);
\draw[thick, ->, dashed]   (v5) edge (v2);
\draw[->, dashed]  (v6) edge (v2);
\draw[thick, ->, dashed]   (v8) edge (v2);
\draw[thick, ->, dashed]   (v9) edge (v2);

\draw[thick]  plot[smooth, tension=.7] coordinates {(v10) (1.2,-1.2) (1.4,-1.6) (v4)};
\draw[thick]  plot[smooth, tension=.7] coordinates {(v11) (1.4,-2.2) (v8)};
\draw[thick]  plot[smooth, tension=.7] coordinates {(v12) (1.6,-3) (v5)};
\draw[thick]  plot[smooth, tension=.7] coordinates {(v13) (1.4,-3.6) (v9)};

\draw  plot[smooth, tension=.7] coordinates {(v10) (1.2,0) (v1)};
\draw  plot[smooth, tension=.7] coordinates {(v11) (1.6,-3) (1.4,-2.2) (v3)};
\draw  plot[smooth, tension=.7] coordinates {(v12) (1.4,-3.6) (2.2,-2.6) (v4)};

\draw[thick,dashed]   plot[smooth, tension=.7] coordinates {(v7) (-1.6,-1) (v10)};
\draw[thick,dashed]  plot[smooth, tension=.7] coordinates {(v7) (-1.4,-1.6) (v11)};
\draw[thick,dashed]  plot[smooth, tension=.7] coordinates {(v7) (-1.4,-2.4) (v12)};
\draw[thick,dashed]  plot[smooth, tension=.7] coordinates {(v7) (-2,-3) (v13)};
\node at (-2,-3.6) {$\mathcal{P}-\{ rv \}$};
\node at (2.2,-1) {$\mathcal{Q}'$};
\node at (0.6,-1.6) {$ \mathcal{P}' $};
\end{tikzpicture}
\caption{The construction of $ \mathcal{P}_n $}\label{mainproof fig}\label{fig mainthm}
\end{figure} 

By applying  Pym's theorem (see Theorem \ref{Pym}) with $ \mathcal{P}' $ and $ \mathcal{Q}' $, we obtain a system $ 
\mathcal{R}' $ of 
disjoint 
$ S\rightarrow N^{\mathsf{in}}_{D-rv}(v_n) $ paths
where $ V_{\text{first}}(\mathcal{R}')=S $ and
$ V_{\text{last}}(\mathcal{R}') $ contains the tails of the edges in $ J $.  We extend  $ \mathcal{R}' $ to a $ v_n $-infan 
 $ \mathcal{R} $  that uses all the edges in $ J $.
Finally we build $ \mathcal{P}_n $ by joining the initial segments of the paths in $ \mathcal{P}-\{ rv \} $ up to $ S $  with 
the 
paths in $ 
\mathcal{R} $ and by adding the path $ \{ rv_n  \}$ if $ rv_n\in D $. The construction ensures that $ 
\mathcal{P}_n\in \mathfrak{I}_D(v_n)  $.
\end{proof}
 
\section{Proof of the lemmas}
\subsection{Preliminaries}\label{ErdMeng sep}
 \begin{thm}[R. Aharoni, E. Berger; Theorem 1.6 of \cite{aharoni2009menger}]\label{inf Menger}
For every digraph $ D $ and $ X,Y\subseteq V $, there is a system $ \mathcal{P} $ of 
disjoint $  X \rightarrow Y $ paths in $ D $  such that one can choose exactly one vertex from each path in $ 
\mathcal{P} $ in such a way that the resulting vertex set $ S $ separates $ Y $ from $ X $ in $ D $. 
\end{thm}
For a rooted digraph $ D $ and $ v\in V-r $, let  $ \boldsymbol{\mathfrak{S}_D(v)} 
$  be the set of the 
those $ S\subseteq V\setminus \{ r,v 
\}  $ 
that separates $ v $ from $ r $ in $ D-rv $ and for which $ D $ admits a system $ \mathcal{P} $  of  internally 
disjoint   $ r\rightarrow v $ paths such that $ S $ consists of choosing exactly one internal vertex from  each 
path in $ \mathcal{P} $. We call $ \mathfrak{S}_D(v) $ the set of the \textbf{Erdős-Menger separations} corresponding to 
$ v $ in the rooted digraph $ D $. By applying Theorem \ref{inf Menger} with $ X:=N_{D-rv}^{\text{out}}(r) $ and 
$ Y:=N_{D-rv}^{\text{in}}(v) $ in $ D-rv $,  we conclude that $  \mathfrak{S}_D(v)\neq \varnothing $ for every rooted 
digraph $ D $ and $ v\in V-r $. Note that if  $ S\in \mathfrak{S}_D(v) $ is 
witnessed by $ \mathcal{P} $, then either $ 
\mathcal{P}\in  \mathfrak{I}_D(v) $ or $ \mathcal{P}+\{ rv \}\in \mathfrak{I}_D(v) $ depending on if $ rv\in D $. Therefore 
$  \mathfrak{I}_D(v)\neq \varnothing $. Furthermore,  if $ S\in \mathfrak{S}_G(v)\cap \mathfrak{S}_D(v)  $ where $ G\subseteq D $ then any 
path-system showing
$S\in \mathfrak{S}_G(v)  $ exemplifies $ S\in \mathfrak{S}_D(v) $ as well.
\begin{cor}\label{reform largeness}
The rooted digraph $ L\subseteq D $ is $ D $-vertex-large if and only if $ L\supseteq\mathsf{out}_{D}(r) $ and
$ \mathfrak{S}_L(v)\cap \mathfrak{S}_D(v)\neq \varnothing$ for every $ v\in V-r $.
\end{cor}

Since for a flame $ F\subseteq D $, $ F\cup \mathsf{out}_D(r) $ remains a flame, finding a $ D $-large flame is equivalent 
with 
finding a flame that preserves an Erdős-Menger separation for each $ v\in V-r $. 

 An augmenting walk for  a system $ \mathcal{P} $ of disjoint $ X\rightarrow Y $ paths in a digraph $ D $ is a finite  $ 
 W\subseteq D $ such that the symmetric difference of $ W $ and $ \bigcup \mathcal{P} $ is (the edge set of) a system of 
 disjoint 
 $ X\rightarrow Y $ paths $ \mathcal{Q} $ covering one more vertex from $ X $ and from $ Y $ than $ \mathcal{P} $. The 
 name comes from the fact that if such a $ W $ exist then it is possible to find one as a walk in a certain auxiliary 
 digraph.  
\begin{lem}[Augmenting walk]\label{aug path thm}
Let $ D $ be a digraph and let $ \mathcal{P}$ be a system of disjoint $ X\rightarrow Y $ paths in $ D $ for some $ 
X,Y\subseteq 
V $. There is either an $ XY $-separation $ S $ consisting of exactly one vertex from each path of $ \mathcal{P} 
$ or
there is a system $ \mathcal{Q} $ of disjoint $ X\rightarrow Y $ paths in $ D $  such that 
$ \left|\mathcal{P}\setminus \mathcal{Q}\right|+1= \left|\mathcal{Q}\setminus \mathcal{P}\right|<\aleph_0$, 
$ V_{\text{first}}(\mathcal{Q})\supseteq V_{\text{first}}(\mathcal{P}) $ and $ V_{\text{last}}(\mathcal{Q})\supseteq 
V_{\text{last}}(\mathcal{P}) $.
 \end{lem}

For more details about the Augmenting walk lemma and its role in the proof of the Aharoni-Berger theorem we refer to 
Lemmas 3.3.2. and 3.3.3. and
Theorem 8.4.2. of \cite{diestel2016graph} (it is discussed for undirected graphs but the directed case involves no additional 
ideas).

It is worth to mention an alternative characterisation of  $  \mathfrak{I}_D(v) $ (follows from Theorem 4.7 of 
\cite{aharoni2009menger}). A system $ \mathcal{P} $ of internally disjoint $ r\rightarrow v $ paths in $ D $ is called
\textbf{strongly 
maximal} if for every internally disjoint system $ \mathcal{Q} $ of $ r\rightarrow v $ paths in $ D $, 
$ \left|\mathcal{Q}\setminus \mathcal{P}\right|\leq \left|\mathcal{P}\setminus \mathcal{Q}\right| $.
\begin{prop}\label{strong max char}
For every rooted digraph $ D $ and $ v\in V-r $, $  \mathfrak{I}_D(v) $ is  the set of the 
strongly maximal internally disjoint $ r\rightarrow v $ path-systems of $ D $. 
\end{prop}

 \begin{proof}
 Assume first $ rv\notin D $.  Let $ \mathcal{P}\in \mathfrak{I}_D(v) $  and pick an $ S 
$ that separates $ v $ from $ r $ and consists of choosing one internal vertex from each path in $ \mathcal{P} $. Let 
$\mathcal{Q}$ be a system of internally disjoint $r\rightarrow v$ paths. Then the 
paths 
$  \mathcal{P}\setminus \mathcal{Q} $ use exactly $   \left|\mathcal{P}\setminus \mathcal{Q}\right| $  vertices from
$ S $  and each path in $ \mathcal{Q}\setminus \mathcal{P} $ goes through  at least one of these vertices. Since the paths $ 
\mathcal{Q}\setminus \mathcal{P} $ are internally disjoint, $ \left|\mathcal{Q}\setminus \mathcal{P}\right|\leq 
\left|\mathcal{P}\setminus \mathcal{Q}\right| $ follows. To show the other direction, let $ \mathcal{P} $ be a strongly 
maximal 
system of internally 
disjoint $ r\rightarrow v $ paths in $ D $. Let $ X:=N_{D}^{\text{out}}(r) $ and 
$ Y:=N_{D}^{\text{in}}(v) $. It is easy to check that the set of the $ X\rightarrow Y $ segments $ \mathcal{P}' $ of the 
paths in $ \mathcal{P} $ form a strongly maximal system of disjoint $ X\rightarrow Y $ paths.  By applying Lemma \ref{aug 
path thm} with $ \mathcal{P}' $, we obtain an $ S $ which exemplifies $ \mathcal{P}\in \mathfrak{I}_D(v) $.

If $ rv\in D $ then $ \mathcal{P} $ is strongly maximal if and only if $ \{ rv \}\in \mathcal{P} $ and $ \mathcal{P}-\{ rv \} $ 
is strongly maximal in $ D-rv $. Since $ \mathcal{P}\in \mathfrak{I}_D(v) $ if and only if $ \{ rv \}\in \mathcal{P} $ and $ 
\mathcal{P}-\{ rv \}\in \mathfrak{I}_{D-rv}(v) $,  we are 
done by applying the proved case in $ D-rv $.
\end{proof}

\subsection{Uniting bubbles}
Let $ D $ be a rooted digraph. The entrance $ \boldsymbol{\mathsf{ent}_D(X)} $ of  an $ X\subseteq V-r $ with respect to $ D 
$ is $ \{ 
 v\in X\, :\, \exists uv\in D\text{ with }u\notin X \} $. We write $ \boldsymbol{\mathsf{int}_D(X)} $ for 
 $ X\setminus \mathsf{ent}_D(X) $. A set $ B\subseteq V-r $ is a $ \boldsymbol{v} $\textbf{-bubble} with respect to $ D $ 
 if  there exists a $ v $-infan 
$ \mathcal{P}=\{ P_{u}: u\in \mathsf{ent}_D(B) \} $  in $ D\cap (B\times B) $ where
$P_u $ starts at $ u $. Let us denote the set of the $ v $-bubbles in $ D $ by $ \boldsymbol{\mathsf{bubb}_D(v)} 
$. Clearly $ \{ v \}\in \mathsf{bubb}_D(v) $ since either the $ v\rightarrow v $ path or the empty set is a witness for it 
depending on if $ 
v\in \mathsf{ent}_D(\{ v \}) $). 
\begin{lem}\label{bubble chain}
Let $ D $ be a rooted digraph and let $ \alpha $ be an ordinal number. Suppose that $ \left\langle B_\beta: \beta<\alpha  
\right\rangle $ is a sequence where 
$ B_\beta\in \mathsf{bubb}_D(v_\beta) $ for some $ v_{\beta}\in V-r $. Let us denote $ \bigcup_{\gamma<\beta}B_\gamma 
$ 
by $ 
\boldsymbol{B_{<\beta}} $.  If 
for each $ \beta<\alpha $ either $ v_\beta=v_0 $ or $ v_\beta\in \mathsf{int}_D \left(B_{<\beta} \right) $, then $ B_{<\alpha}\in 
\mathsf{bubb}_D(v_0) $.  
\end{lem}
\begin{sbiz}
For every  $u\in \bigcup_{1\leq \beta \leq \alpha} \mathsf{ent}_D(B_{<\beta}) $, 
we 
construct a  $u\rightarrow v_0  $ 
path $ P_u $ in such a way that for each $ \beta  $ the paths $ \{ P_u\, : \, u\in \mathsf{ent}_D(B_{<\beta})   \} $ exemplify 
$ 
B_{<\beta}\in 
\mathsf{bubb}_D(v_0) $.

Let $ \{ P_u\, : \, u\in \mathsf{ent}_D(B_{<1})   \}$ be an arbitrary path-system witnessing $ B_{0}\in 
\mathsf{bubb}_D(v_0)  $. Suppose that $ \beta>1 $ and $ P_u $ is defined whenever there is a $ \gamma<\beta $ for which 
$ u\in \mathsf{ent}_D(B_{<\gamma}) $. If $ \beta $ is a limit ordinal then  $ P_u $ is defined for 
$ u\in \mathsf{ent}_D(B_{<\beta}) $ as well and the conditions hold. 

Assume that $ 
\beta=\gamma+1 $.  Fix a path-system $ \mathcal{Q}=\{ Q_u: u\in \mathsf{ent}_D(B_\gamma) \} $ which shows $ 
B_\gamma\in 
\mathsf{bubb}_D(v_\gamma) $. 
Note that $ \mathsf{ent}_D(B_{<\gamma+1})\setminus\mathsf{ent}_D(B_{<\gamma})  \subseteq 
\mathsf{ent}_D(B_\gamma) $ (see Figure \ref{unite bubb fig}). Since $ v_\gamma\in \{ v_0 \}\cup 
\mathsf{int}_D(B_{<\gamma}) $,  all the 
paths in $ 
\mathcal{Q} $ meet $ B_{<\gamma} $. Furthermore, if two paths in $ \mathcal{Q} $ have the same 
vertex  as first meeting with $ B_{<\gamma} $, then  it must be $ v_\gamma $ and hence $ v_\gamma=v_0 $ holds  since 
in 
this case 
$ v_\gamma \notin \mathsf{int}_D(B_{<\gamma})  $. For $ u\in 
\mathsf{ent}_D(B_{<\gamma+1})\setminus\mathsf{ent}_D(B_{<\gamma}) $, consider the initial 
segment $ Q_u' $ of $ Q_u $ up 
to the first 
 vertex $ w $ that is in $ B_{<\gamma} $. Join $ Q_u' $ and $ P_w $ to obtain $ P_u $.
\end{sbiz}
\begin{figure}[H]
\centering
\begin{tikzpicture}

\draw  (-3.5,1) node (v1) {} rectangle (2.5,-1.5);
\draw  (v1) rectangle (2.5,0.5);

\draw[dashed]   (2.5,1) rectangle (0,3) node (v2) {};
\draw[dashed]  (v2) rectangle (2.5,2.5);

\node (v7) at (-2.6,0.8) {};
\node (v9) at (-1,0.8) {};
\node (v4) at (0.2,0.8) {};
\node (v6) at (1.2,0.8) {$w$};
\node (v10) at (2,0.8) {};
\node (v8) at (-0.6,-1) {$v_0$};
\node (v3) at (0.6,2.8) {};
\node (v5) at (1.8,2.8) {$u$};

\node at (1.5,1.6) {$Q'_u$};
\node at (0.5,-0.2) {$P_w$};

\node at (-4.8,-0.4) {$B_{<\gamma}$};
\node at (-4.4,0.6) {$\mathsf{ent}_D(B_{<\gamma})$};
\node at (-1.8,2) {$B_{<\gamma+1}\setminus B_{<\gamma}$};
\node at (-2.2,2.8) {$\mathsf{ent}_D(B_{<\gamma+1})\setminus \mathsf{ent}_D(B_{<\gamma})$};
\draw  plot[smooth, tension=.7] coordinates {(v3) (0.8,2.2) (0.4,1.6) (v4)};
\draw  plot[smooth, tension=.7] coordinates {(v5) (1.6,2) (1.2,1.8) (1,1.4) (v6)};

\draw[thick]  plot[smooth, tension=.7] coordinates {(v7) (-2.4,-0.4) (-1.8,-0.6) (-1,-0.8) (v8)};
\draw[thick]  plot[smooth, tension=.7] coordinates {(v9)};
\draw[thick]  plot[smooth, tension=.7] coordinates {(v9) (-1,0) (v8)};
\draw[thick]  plot[smooth, tension=.7] coordinates {(v4) (0,-0.2) (v8)};
\draw[thick]  plot[smooth, tension=.7] coordinates {(v6)};
\draw[thick]  plot[smooth, tension=.7] coordinates {(v6) (0.6,-0.6) (v8)};
\draw[thick]  plot[smooth, tension=.7] coordinates {(v10) (1.4,-0.2) (1,-1) (v8)};
\end{tikzpicture}
\caption{The construction of the path-system witnessing  $ B_{<\gamma+1}\in \mathsf{bubb}_D(v_0) $. (Vertex $ v_0 $ 
can be in  $  \mathsf{ent}_D(B_{<\gamma}) $)}\label{unite bubb fig}
\end{figure}
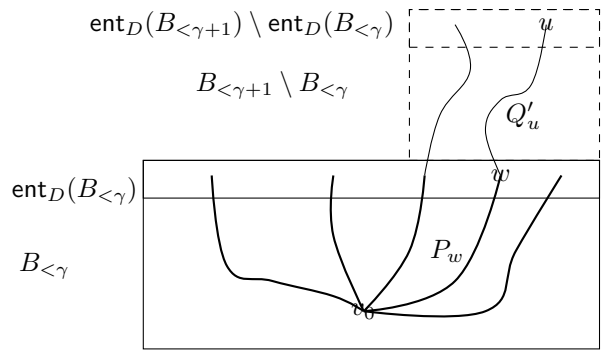 
\begin{cor}\label{bubbles closed union}
For every rooted digraph $ D $ and $ v\in V-r $, $\mathsf{bubb}_D(v) $ is closed under arbitrary large 
union.
\end{cor}

\begin{cor}\label{bigest bubble}
For every rooted digraph $ D $ and  $ v\in V-r $, there is a $ \subseteq $-largest element of $ \mathsf{bubb}_D(v) $, 
namely $ \bigcup \mathsf{bubb}_D(v)=:\boldsymbol{B_{v,D}} $.  
\end{cor}

\subsection{Preserving largeness}

For $ S\in \mathfrak{S}_D(v) $, we denote by $ \boldsymbol{B_{S,v,D}} $ the 
set of those  $ u\in V $ for which every 
$ r\rightarrow u $ path in  $ D-rv $ meets $ S $.
\begin{prop}\label{B_CE hatara}
Assume that  $ D $ is a rooted digraph, $ v\in V-r$ and $ S\in \mathfrak{S}_D(v)$. Then $ B_{S,v,D}\in 
\mathsf{bubb}_D(v) $ with
 $ \mathsf{ent}_{D-rv}(B_{S,v,D})=S $ and $ N^{\text{in}}_{D-rv}(v) \subseteq B_{S,v,D}$.
\end{prop}
\begin{sbiz}
 The inclusions
 $  B_{S,v,D} \supseteq S\supseteq\mathsf{ent}_{D-rv}(B_{S,v,D}) $  are clear from the definition of $ B_{S,v,D} $ as well 
 as $ r\notin  B_{S,v,D} $ and $ v\in B_{S,v,D} $. Suppose for a 
  contradiction
  that 
  there is a $ u\in S \setminus \mathsf{ent}_{D-rv}(B_{S,v,D})  $. By 
 the choice of 
  $ S $,  there is a system $ \mathcal{P} $ of internally disjoint $ r\rightarrow v $ paths in 
 $ D-rv $ for which $ S $ consists of choosing one internal vertex from each path in $ \mathcal{P} $. 
 The unique path $ P_u\in \mathcal{P} $ which goes through $ u $ enters $ B_{S,v,D} $ and therefore meets $ 
 \mathsf{ent}_{D-rv}(B_{S,v,D}) $. But then $ P_u $ has at least two vertices in $ S $ since $ 
 S\supseteq\mathsf{ent}_{D-rv}(B_{S,v,D}) $ which is a contradiction. Hence $ S=\mathsf{ent}_{D-rv}(B_{S,v,D}) $. 
 The terminal segments  of the paths in $ \mathcal{P} $  from $ S $  witnessing $ B_{S,v,D}\in 
 \mathsf{bubb}_D(v) $. Finally if $ w\in N^{\text{in}}_{D-rv}(v) \setminus B_{S,v,D} $ then $ w $ is reachable from 
 $ r $ without touching $ S $ and hence $ v $ as well (since $ v\notin S $ ) which is impossible.
\end{sbiz}

\begin{prop}\label{nagybubi hatar}
For every rooted digraph $ D $ and $ v\in V-r $,  \[ \mathsf{ent}_{D-rv}(B_{v,D})\in 
\mathfrak{S}_D(v). \]
\end{prop}
\begin{sbiz}
Let $ S\in \mathfrak{S}_D(v) $ be arbitrary.  It follows from Proposition \ref{B_CE 
hatara} that $ B_{v,D}\supseteq S $. Thus $ \mathsf{ent}_{D-rv}(B_{v,D}) $ separates $ S $ from $ r $ in $ D-rv $ and 
therefore it
separates $ v $ from $ r $ in $ D-rv $ as well. It remains to show a system $ \{ P_u: u\in  \mathsf{ent}_{D-rv}(B_{v,D})\} $ of internally disjoint $ 
r\rightarrow v $ paths where $ P_u $ goes through $ u $.  Since $ B_{v,D}\in 
\mathsf{bubb}_D(v) $, it is enough to prove that there is an $ r $-fan $ \{ Q_u\, :\ u\in \mathsf{ent}_{D-rv}(B_{v,D}) \} $  in $ 
D-rv $ where $ Q_u $ terminates at $ u $. To do so, apply the Aharoni-Berger theorem (Theorem 
\ref{inf Menger}) in $ D-rv $ with $ X:=N^{\text{out}}_{D-rv}(r) $ and $ Y:=\mathsf{ent}_{D-rv}(B_{v,D}) $ (see Figure 
\ref{fig entrance and u} without the paths). Observe that 
the resulting separation 
$ T $ is in $ \mathfrak{S}_D(v) $. By Proposition \ref{B_CE hatara}, $ B_{T,v,D}\in 
\mathsf{bubb}_D(v) $, which implies 
$ B_{T,v,D}\subseteq B_{v,D} $. It follows that $ T\subseteq \mathsf{ent}_{D-rv}(B_{v,D}) $. Suppose for a contradiction 
that 
$ w\in \mathsf{ent}_{D-rv}(B_{v,D})\setminus T $. Since $ w\in \mathsf{ent}_{D-rv}(B_{v,D}) $, we can pick a $ uw 
\in D-rv$ with $ u\notin B_{v,D} $. Since 
$ T $ separates $ w \in  \mathsf{ent}_{D-rv}(B_{v,D}) $ from $ X  $ in $ D-rv $ and $ w\notin T $, we 
may conclude that $ T $ 
separates $ u $ from $ X  $ in $ D-rv $ as well. But then $ u\in B_{T,v,D}\subseteq B_{v,D} $ 
contradicts
$ u\notin B_{v,D} $. 
\end{sbiz}
\begin{lem}[Characterisation of largeness]\label{char of largeness}
Let  $ L\subseteq D $ be rooted digraphs.  
Then $ L$ is $ D $-vertex-large if and 
only if  $ u\in B_{v,L} $
for every $ uv\in D\setminus L$. 
Furthermore, if $ L $  is $ D $-vertex-large and $ v\in V-r $, then
$ \mathsf{ent}_{D-rv}(B_{v,L})=\mathsf{ent}_{L-rv}(B_{v,L})\in  \mathfrak{S}_D(v) $.
\end{lem}
\begin{sbiz}
Assume that $ L $ is $ D $-large and let $ uv\in D\setminus L $. By applying the reformulation of largeness in Corollary \ref{reform largeness} we 
know that $ u\neq r $ and there is some $ S\in \mathfrak{S}_L(v)\cap \mathfrak{S}_D(v)  $. Proposition \ref{B_CE hatara} guarantees $ u\in 
B_{S,v,D}$. Since $ L\subseteq D $, it follows directly from the definition that  $ B_{S,v,D}\subseteq B_{S,v,L} $.
By combining these,  we obtain $  u\in 
B_{S,v,D}\subseteq B_{S,v,L}\subseteq B_{v,L} $.
 
For the ``if'' direction, take an arbitrary $ v\in V-r 
$. 
If $ rv\in D $, then $ rv\in L $ because $ r\in B_{v,L} $ is impossible by 
the definition of bubbles. By 
Proposition \ref{nagybubi hatar}, we know that $ \mathsf{ent}_{L-rv}(B_{v,L})\in \mathfrak{S}_L(v)  $. Suppose for 
contradiction that $ \mathsf{ent}_{L-rv}(B_{v,L})\notin \mathfrak{S}_D(v) $.  The only possible reason 
for this is that $ \mathsf{ent}_{L-rv}(B_{v,L}) $ does not separate $ v $ 
from $ r $ in $ D-rv $ (just in $ L-rv $).
It follows that there is a $ w\in \mathsf{ent}_{D-rv}(B_{v,L})\setminus \mathsf{ent}_{L-rv}(B_{v,L}) $ witnessed by some 
edge $ uw\in D\setminus L $.  Then $ w\in \mathsf{int}_{L-rv}(B_{v,L}) $ and by assumption $ u\in B_{w,L} $. Hence by 
Lemma \ref{bubble chain}, $ B_{v,L}\cup B_{w,L}\in \mathsf{bubb}_L(v) $ which is a contradiction because $ u $ 
witnesses 
 $B_{v,L}\subsetneq (B_{v,L}\cup B_{w,L}) $.
\end{sbiz}
\begin{proof}[Proof of Lemma \ref{key lemma}]
To show the largeness of $ L $, we use the characterization of largeness in Lemma \ref{char of largeness}. Let  $ uv\in 
D\setminus L $ be arbitrary. By  assumption there is a $ \mathcal{P}\in 
\mathfrak{I}_D(v) $ 
that lies in $ 
L $. We pick an $ S $ witnessing $ \mathcal{P}\in 
\mathfrak{I}_D(v)  $. By Proposition \ref{B_CE hatara},  $ u\in B_{S,v,D}$. Therefore $u\in B_{S,v,D}\subseteq B_{S,v,L}\subseteq B_{v,L}  $.
Since $ uv\in D\setminus L $ was arbitrary,  we may conclude by Lemma \ref{char of largeness} that $ L $ is large.
\end{proof}
\subsection{A ``largeness-faithful'' quasi-flame}
\begin{claim}\label{one more path}
For every rooted digraph $ D $,  $ v\in V-r $ and $ u\in (V\setminus B_{v,D})-r $, there is an $ r $-fan $ \mathcal{P} $  in $ D-rv 
$ with $ V_{\text{last}}(\mathcal{P})=\mathsf{ent}_{D-rv}(B_{v,D})+u $.
\end{claim}
\begin{sbiz}
It follows from Proposition \ref{nagybubi hatar} that there is an $ r $-fan $ \mathcal{Q} $ with $ 
V_{\text{last}}(\mathcal{Q})=\mathsf{ent}_{D-rv}(B_{v,D}) $. 
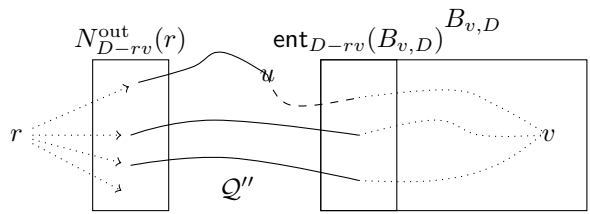
\begin{figure}[H]
\centering
\begin{tikzpicture}

\node (v3) at (-2.5,0) {$r$};
\draw  (-1.5,1) rectangle (-0.5,-1);
\draw  (1.5,1) node (v1) {} rectangle (5,-1);
\draw  (v1) rectangle (2.5,-1);
\node (v2) at (4.5,0) {$v$};
\draw[dotted]  plot[smooth, dotted,  tension=.7] coordinates {(2,0.5) (3.2,0.6) (3.8,0.4) (v2)};
\draw[dotted]   plot[smooth, tension=.7] coordinates {(2,0) (3,0.2) (3.6,0) (4.4,0)};
\draw[dotted]   plot[smooth, tension=.7] coordinates {(2,-0.6) (3,-0.6) (3.8,-0.4) (4.4,0)};

\node (v4) at (-0.9,0.7) {};
\node (v5) at (-1,0) {};
\node (v6) at (-1,-0.4) {};
\node (v7) at (-1,-0.8) {};
\draw  (v3) edge[dotted, ->] (v4);
\draw  (v3) edge[dotted, ->] (v5);
\draw  (v3) edge[dotted, ->] (v6);
\draw  (v3) edge[dotted, ->] (v7);

\node at (-1,1.25) {$N^{\text{out}}_{D-rv}(r)$};
\node at (2,1.25) {$\mathsf{ent}_{D-rv}(B_{v,D})$};
\node at (3.5,1.5) {$B_{v,D}$};

\node (v8) at (0.8,0.8) {$u$};

\draw[dashed]  plot[smooth, tension=.7] coordinates {(v8) (1.1,0.4) (2,0.5)};
\draw  plot[smooth, tension=.7] coordinates {(v4) (-0.2,0.9) (0.2,1.1) (v8)};
\draw  plot[smooth, tension=.7] coordinates {(v5) (0.1,0.2) (2,0)};
\draw  plot[smooth, tension=.7] coordinates {(v6) (0.3,-0.3) (2,-0.6)};
\node at (0.4,-0.7) {$ \mathcal{Q}'' $};
\end{tikzpicture}
\caption{The construction of $ \mathcal{P} $}\label{fig entrance and u}
\end{figure} 
Let $ \mathcal{Q}':=\mathcal{Q} $ if $ \mathcal{Q} $ does 
not meet $ u $ otherwise replace the unique $ Q\in \mathcal{Q} $ through $ u $ by its initial segment up to $ u $ to obtain 
$ 
\mathcal{Q}' $ from $ \mathcal{Q} $. Either way  $ \mathcal{Q}' $ is an $ r $-fan such that $ V_{\text{last}}(\mathcal{Q}') 
$ 
consists of all but one element of  $  \mathsf{ent}_{D-rv}(B_{v,D})+u=:Y $ and the paths $ \mathcal{Q}' $ can be extended 
forward to get a internally disjoint system of $ r\rightarrow v $ paths. Let $ X:=N^{\text{out}}_{D-rv}(r) $ and let $ 
\mathcal{Q}'' $ consist of the  $ X\rightarrow Y $  (terminal) segments of the paths $ \mathcal{Q}' $. We apply  the
Augmenting walk method (Lemma \ref{aug path thm}) with $ \mathcal{Q}'' $ in $ D-rv $. If the augmentation is possible 
then the resulting path-system together with the corresponding edges from $ r $ is appropriate for $ \mathcal{P} $. Suppose 
for a contradiction that the
augmentation is impossible.Then we obtain an $ S\in \mathfrak{S}_D(v) $ that separates $ 
\mathsf{ent}_{D-rv}(B_{v,D})+u $ from 
$ N^{\text{out}}_{D-rv}(r) $
in $ D-rv $.  By the choice of $ u $, $ B_{S,v,D}\supsetneq B_{v,D} $ which is a contradiction.
\end{sbiz}

\begin{lem}\label{berak coloop edge}
Let $ D\supseteq G\supseteq H $ be rooted digraphs and let $ v\in V-r $. Suppose that there is a
$ uw\in D\setminus G $ with $ u\notin B_{v,H} $ and $ w\in \mathsf{int}_{H-rv}(B_{v,H}) $. If
$ \mathsf{ent}_{H-rv}(B_{v,H})= \mathsf{ent}_{G-rv}(B_{v,H}) $, then 
$ (I+uw)\in \mathcal{G}_{G+uw}(w) $ for all $ I\in \mathcal{G}_{G}(w) $.
\end{lem}
\begin{sbiz}
The statement is trivial for $ u=r $ thus let $ u\neq r $. We may assume that $ rw\notin I $ otherwise we apply the Lemma first with $ I-rw $ and then 
adding  $ rw 
$ cannot ruin anything . By Claim \ref{one more path}, there is an $ r $-fan $ 
\mathcal{P} $  
  in $ H-rv $ with $ V_{\text{last}}(\mathcal{P})=\mathsf{ent}_{H-rv}(B_{v,H})+u $. Then $ 
\mathcal{P} $ is an $ r $-fan 
 in $ G-rv $ for which $ V_{\text{last}}(\mathcal{P})=\mathsf{ent}_{G-rv}(B_{v,H})+u $ because
$ G \supseteq H $ and $ \mathsf{ent}_{H-rv}(B_{v,H})= \mathsf{ent}_{G-rv}(B_{v,H}) $ by assumption. Let $ \mathcal{Q} 
$ be a path-system witnessing $ I\in \mathcal{G}_G(w) $. Continue forward the paths in $ \mathcal{P} $ (see Figure 
\ref{single-edge extension}) using the terminal 
segments $ \mathcal{Q}' $ of the paths in $ \mathcal{Q} $ from the last intersection with $ \mathsf{ent}_{G-rv}(B_{v,H})  
$ and edge $ uw 
$ to obtain a path-system
witnessing $ (I+uw)\in \mathcal{G}_{G+uw}(w) $.
\end{sbiz}
\begin{figure}[H]
\centering
\begin{tikzpicture}

\draw  (-0.5,-1) node (v1) {} rectangle (3.5,-4.5);
\draw  (v1) rectangle (0,-3);
\node (v6) at (-0.2,-1.4) {};
\node (v3) at (-0.2,-1.8) {};
\node (v4) at (-0.2,-2.2) {};
\node (v2) at (-0.2,-2.6) {};
\node (v8) at (-1.2,-4) {$u$};
\node (v5) at (0,-3.8) {$w$};
\node (v7) at (-2.6,-2.6) {};
\node (v11) at (-2.9,-2.6) {$ r $};
\draw  (v8) edge[->,dashed] (v5);
\node (v9) at (-1.8,-1.2) {$ \mathsf{ent}_{G-rv}(B_{v,H}) $};
\node (v10) at (1.5,-0.8) {$ B_{v,H}$};

\node (v12) at (-2,-3.8) {$ \mathcal{P} $};
\node (v13) at (1.5,-3) {$ \mathcal{Q}' $};

\draw  plot[smooth, tension=.7] coordinates {(v2) (0.4,-2.6) (0.8,-3.2) (0.4,-3.6) (v5)};

\draw plot[smooth, tension=.7] coordinates {(v3) (0.8,-2) (1.4,-2.2) (2,-3) (1.6,-3.4) (0.8,-4) (v5)};

\draw  plot[smooth, tension=.7] coordinates {(v6) (1,-1.6) (2.2,-1.6) (2.6,-2) (2.6,-3) (2.2,-3.8) (1.6,-4.2) (0.2,-4.2) (v5)};

\draw[thick]    plot[smooth, tension=.7] coordinates {(v7) (-2,-2) (-1.2,-1.6) (v6)};

\draw[thick]    plot[smooth, tension=.7] coordinates {(v7) (-1.4,-2.4) (-1,-1.8) (v3)};

\draw[thick]    plot[smooth, tension=.7] coordinates {(v7) (-2,-3) (-1.4,-3) (-1,-2.4) (v4)};

\draw[thick]    plot[smooth, tension=.7] coordinates {(v7) (-2.6,-3.2) (-2.2,-3.4) (-1.6,-3.6) (-1.2,-3.2) (v2)};

\draw[thick]    plot[smooth, tension=.7] coordinates {(v7) (-3,-3.2) (-2.6,-3.8) (-2,-4.2) (v8)};
\end{tikzpicture}
\caption{Extending $ G $ by $ uw $. }\label{single-edge extension}
\end{figure}
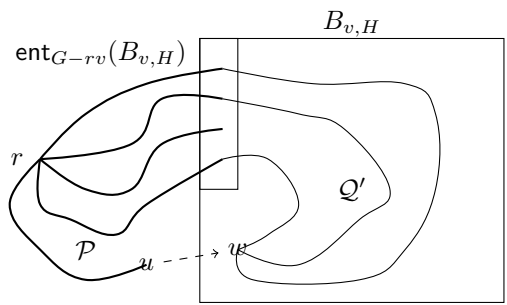

\begin{lem}\label{superlarge}
Let $ D\supseteq G $ be rooted digraphs. If for every $ uv\in D\setminus G $ there is an $ I\in 
\mathcal{G}_G(v) $ such that 
$ (I+uv)\notin \mathcal{G}_{G+uv}(v)  $, then whenever an $ H$ is $G $-vertex-large it is $ D $-vertex-large as well.
\end{lem}
\begin{sbiz}
Suppose for a contradiction that $ H $ is $ G $-large but not $ D $-large. By applying Lemma \ref{char of 
largeness} with $ D $ and $ H $, we 
obtain that there is some 
$ uv\in D\setminus H $ with $ u\notin B_{v,H} $. Since $ H $ is assumed to be $G $-large, we may conclude by Lemma 
\ref{char of 
largeness} that $ uv\notin G $ 
and $ \mathsf{ent}_{G-rv}(B_{v,H})=\mathsf{ent}_{H-rv}(B_{v,H}) $. 
Finally $ v\in 
\mathsf{int}_{H-rv}(B_{v,H}) $ 
by Proposition \ref{nagybubi hatar}. We use Lemma \ref{berak 
coloop edge} with $ w:=v $ 
to obtain
$ (I+uv)\in \mathcal{G}_{G+uv}(v) $ for every $ I\in 
\mathcal{G}_G(v) $  which is a contradiction.
\end{sbiz}

\begin{proof}[Proof of Lemma \ref{large quasi flame}]
By applying Zorn's lemma, we may pick a $ \subseteq $-maximal quasi-flame $ F\subseteq D $. Lemma \ref{superlarge} 
ensures that  whenever an $ L $ is $ F $-large it is $ D $ -large as well.
\end{proof}
\subsection{Preserving the quasi-flame property via preserving largeness}\label{subsec prev quasi}
\begin{claim}\label{reach finite preserved}
Let $ D $ be a rooted digraph and let $ X\subseteq V-r $ be finite. Suppose that there is an $ r $-fan $ \mathcal{P} $ in $ D $ 
with $ V_{\text{last}}(\mathcal{P})=X $. If $ L $ is $ D $-vertex-large, then there is  an $ r $-fan $ 
\mathcal{Q} $ in $ L $ with $ V_{\text{last}}(\mathcal{Q})=X $. 
\end{claim}
\begin{sbiz}
Suppose for a contradiction that $ L $ does not contain a desired $ r $-fan $ \mathcal{Q} $. We may assume that there is a $ 
v\in V $ 
which is 
an isolated vertex in $ D $  (otherwise we consider  an isomorphic copy of $ D$  on a proper subset of $ V $). Extend $ D 
$ and $ 
L $ with  
the  edges $ xv\ (x\in X) $ to obtain $ D' $ and $ L' $ respectively. Clearly
$ \kappa_{L'}(r,v)<\kappa_{D'}(r,v)=\left|X\right| $. Since $ rv\notin D' $, Proposition \ref{nagybubi hatar} ensures $ 
\kappa_{L'}(r,v)= \left|\mathsf{ent}_{L'}(B_{v,L'})  \right| $.  By combining these,
$\kappa_{L'}(r,v)= \left|\mathsf{ent}_{L'}(B_{v,L'})  \right|<\left|X\right|<\aleph_0 $ (see Figure \ref{Xeler}). It follows that 
there is a $ uw\in 
D\setminus L $ with  $u\notin B_{v,L'} $ and $ w\in  \mathsf{int}_{L'}(B_{v,L'}) $ otherwise  
$\mathsf{ent}_{L'}(B_{v,L'})  $ would separate $ v $ from $ r $ in $ D' $ as well. 

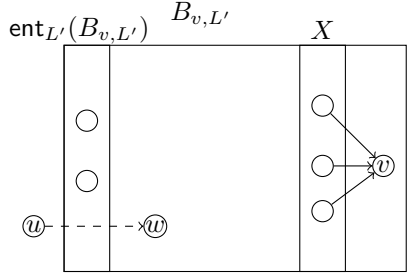
\begin{figure}[H]
\centering
\begin{tikzpicture}

\draw  (-2,2) node (v1) {} rectangle (2.5,-1);
\draw  (v1) rectangle (-1.4,-1);
\draw  (1.1,2) rectangle (1.7,-1.);

\node[circle,inner sep=0pt,draw,minimum size=8] (v3) at (2.2,0.4) {$v$};
\node[circle,inner sep=0pt,draw,minimum size=8] (v2) at (1.4,1.2) {};
\node[circle,inner sep=0pt,draw,minimum size=8] (v4) at (1.4,0.4) {};
\node[circle,inner sep=0pt,draw,minimum size=8] (v5) at (1.4,-0.2) {};

\node[circle,inner sep=0pt,draw,minimum size=8] at (-1.7,1) {};
\node[circle,inner sep=0pt,draw,minimum size=8] at (-1.7,0.2) {};
\node[circle,inner sep=0pt,draw,minimum size=8] (v7) at (-0.8,-0.4) {$w$};
\node[circle,inner sep=0pt,draw,minimum size=8] (v6) at (-2.4,-0.4) {$u$};

\node at (-1.8,2.2) {$\mathsf{ent}_{L'}(B_{v,L'})  $};
\node at (1.4,2.2) {$X$};
\node at (-0.2,2.4) {$B_{v,L'}$};

\draw  (v2) edge[->] (v3);
\draw  (v4) edge[->] (v3);
\draw  (v5) edge[->] (v3);

\draw  (v6) edge[->, dashed] (v7);
\end{tikzpicture}
\caption{$ B_{v,L'} $ and $ uw\in D\setminus L $ ($ X $ may have common vertices with $\mathsf{ent}_{L'}(B_{v,L'})  $) 
}\label{Xeler}
\end{figure} 

By applying  \ref{berak coloop edge} with $ 
D' $ and $ G:=H:=L' $, we 
may 
conclude that $ (I+uw)\in \mathcal{G}_{L'+uw}(w) $ for all $ I\in \mathcal{G}_{L'}(w) $. From the 
construction is clear that $ v $ has no outgoing edges and $ w\neq v $ therefore $ (I+uw)\in \mathcal{G}_{L+uw}(w) $ holds 
for all $ I\in 
\mathcal{G}_{L}(w) $  as well. The finite separation  $\mathsf{ent}_{L'}(B_{v,L'})  $ witnesses that $  
\kappa_{L}(r,w)<\aleph_0 $. Since $ L $ is 
$ D $-large,  $  \kappa_{D}(r,w)=\kappa_{L}(r,w)<\aleph_0 $. Take a system $ \mathcal{P} $ of internally 
disjoint $ r\rightarrow w $ paths in $ L $ with $ \left|\mathcal{P}\right|=\kappa_{D}(r,w) $. Then for  $ 
I=A_{\text{last}}({\mathcal{P}})  $, we obtain $ (A_{\text{last}}({\mathcal{P}}) +uw)\in \mathcal{G}_{L+uw}(w) 
$, which implies  that $ L $ contains a system of 
internally disjoint $ r\rightarrow w $ paths of size $ \left|A_{\text{last}}({\mathcal{P}}) 
\right|+1=\kappa_{D}(r,v)+1<\aleph_0 
$ which is a 
contradiction.
\end{sbiz}

\begin{proof}[Proof of Lemma \ref{key lemma1}]
Suppose that $ D $ is a quasi-flame and $ L $ is $ D $-large. Let $ v\in V-r $ be arbitrary. Assume first that $ 
\kappa_{D}(r,v)<\aleph_0 $. Then $ 
\kappa_{D}(r,v)=\left|\mathsf{in}_{D}(v)\right| $ follows from the 
fact that $ D $ is a quasi-flame. By the largeness of $ L $, we have $ \kappa_{L}(r,v)= \kappa_{D}(r,v) $. By combining 
these, 
we may conclude $ \kappa_{L}(r,v)=\left|\mathsf{in}_{D}(v)\right|<\aleph_0 $ which means 
$\mathsf{in}_{D}(v)=\mathsf{in}_{L}(v)\in \mathcal{G}_{L}(v)  $.

Suppose now that $ \kappa_{D}(r,v)\geq\aleph_0 $ and let $ J\subseteq \mathsf{in}_L(v) $ be finite.  Since $ D $ is a 
quasi-flame we can pick a $ \mathcal{P} $ witnessing 
$ J\in \mathcal{G}_{D}(v) $. Let $ \mathcal{P}' $  be the $ r $-fan that we get by the deletion of the last edges of the paths $ 
\mathcal{P}  $. Since none of the paths in $ \mathcal{P}' $ goes through $ v $ and $ \kappa_{D}(r,v)\geq\aleph_0 $, we can 
extend $ \mathcal{P}' $  in $ D $ by a new path to obtain an $ r $-fan $ \mathcal{P}'' $ where $ 
V_{\text{last}}(\mathcal{P}'')=:X $  consists of 
 the tails of the edges in $ J $ and $ v 
 $ . By Claim \ref{reach finite preserved}, there is  an $ r $-fan $ \mathcal{Q}'' $ 
in $ L $ with $ V_{\text{last}}(\mathcal{Q}'')=X $.  To obtain 
$ Q' $, we delete the unique path in $ \mathcal{Q} $ which terminates 
$ v $. Then none of the paths in $ Q' $ goes through $ v $ and hence by extending them with the edges $ J $, the resulting 
path-system $ \mathcal{Q} $ witnesses $ J\in 
\mathcal{G}_{L}(v) $.
\end{proof}

\section{Open problems}

\subsection{Beyond countability}
One can replace in Theorem \ref{main result flame} the countability of $ D $ by the formally weaker 
assumption that $ \kappa_D(r,v)\leq \aleph_0 $ for every $ v\in V-r $ 
(it is an easy application of Davies-trees  \cite{soukup2014davies}). We believe that more is true.

\begin{conj}
We may omit the countability of $ D $ in Theorem \ref{main result flame}.   
\end{conj}

\subsection{Upper and lower bounds}
\begin{ques}
Let $ F\subseteq L\subseteq D $ be a rooted digraphs where $ F $ is a vertex-flame and $ L $ is $ D 
$-vertex-large. Is there necessarily an $ E $ with $ F\subseteq E\subseteq L $ such that $ E $ is a vertex-flame which is $ D 
$-vertex-large?
\end{ques}
It is true if $ \kappa_D(r,v)<\aleph_0 $ for $ v\in V-r $. We can simply take a $\subseteq  $-maximal quasi-flame in $ 
L $ which 
extends $ F $. It is automatically a flame and its $ L $-largeness (see the proof before subsection \ref{subsec prev quasi}) implies 
that it is $ D 
$-vertex-large as well. 
\subsection{Preserving all Erdős-Menger separations in a flame}
 Consider the reformulation of $ D $-largeness in Corollary \ref{reform largeness}. If 
$ \kappa_D(r,v)<\aleph_0 $ then $ \mathfrak{S}_D(v)\cap \mathfrak{S}_F(v)\neq \varnothing $ implies $ 
\mathfrak{S}_D(v)\subseteq \mathfrak{S}_F(v) $. Thus in a finite rooted digraph we  preserve automatically every 
Erdős-Menger separation when we demand $ D $-largeness but it is usually false for infinite digraphs. It seems a natural 
question 
if we can always preserve all the Erdős-Menger separations in a flame. 
We construct a rooted digraph $ D $ of size $ 2^{\aleph_0} $ which witnesses that the answer is no (however the question 
remains open for smaller digraphs).
 
Consider the digraph at 
Figure \ref{counterexample}. Extend it with all the edges  $ v_iv_{j,k}\   (i\leq \omega,\, j<\omega,\, k<2) $. For every
$ f\in {}^{\omega}2  $, pick a new vertex $ v_f $ with 
$ N_{D}^{\text{in}}(v_f):=\{ v_{i,f(i)}\, :\, i<\omega \} $ and $ N_{D}^{\text{out}}(v_f)=\varnothing $. It is easy to check 
that $ 
\{ v_{i,0}v_i,v_{i,1}v_i  \}\notin \mathcal{G}_D(v_i) $  for $ i<\omega $ and 
$ N_D^{\text{out}}(r)\in \mathfrak{S}_D(v_f)  $ for each 
$ f\in {}^{\omega}2 $. Let   $ F\subseteq D $ be a flame. Then there is an 
$ f\in {}^{\omega}2 $ such that $ v_{i,1-f(i)}v_i\notin F $ for $ i<\omega $. We show that 
$ N_D^{\text{out}}(r)\notin \mathfrak{S}_F(v_f) $. Suppose for a contradiction that 
$ \mathcal{P}:=\{ P_{w}: w\in N_D^{\text{out}}(r) \} $ exemplifies 
that 
$ N_D^{\text{out}}(r)\in \mathfrak{S}_F(v_f) $  where $ P_{w} $ goes through $ w $. Then path $ P_{u_i} $ necessarily 
uses vertex $ v_{i,f(i)} $ because $ u_iv_{i,1-f(i)}\notin F $. But then the paths $ 
P_{u_i} $ use already all the vertices  $\{ 
v_{i,f(i)}\, :\, i<\omega \}= N_{D}^{\text{in}}(v_f) $ and therefore $ P_{v_\omega} $ has no chance to reach $ v_f $ 
which is a contradiction.
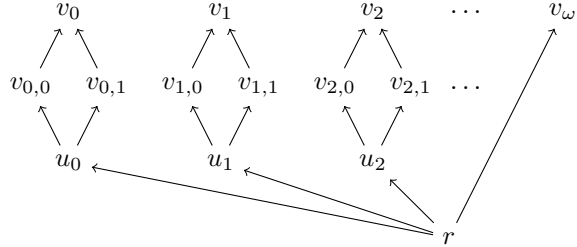
\begin{figure}[H]
\centering

\begin{tikzpicture}

\node (v13) at (3,-1) {$r$};
\node (v1) at (-2,0) {$u_0$};
\node (v8) at (0,0) {$u_1$};
\node (v9) at (2,0) {$u_2$};

\node (v2) at (-2.5,1) {$v_{0,0}$};
\node (v3) at (-1.5,1) {$v_{0,1}$};
\node (v5) at (-0.5,1) {$v_{1,0}$};
\node (v7) at (0.5,1) {$v_{1,1}$};
\node (v10) at (1.5,1) {$v_{2,0}$};
\node (v12) at (2.5,1) {$v_{2,1}$};
\node at (3.25,1) {$\dots$};
\node at (3.25,2) {$\dots$};

\node (v4) at (-2,2) {$v_0$};
\node (v6) at (0,2) {$v_1$};
\node (v11) at (2,2) {$v_2$};
\node (v14) at (4.5,2) {$v_{\omega}$};

\draw  (v1) edge[->] (v2);
\draw  (v1) edge[->] (v3);
\draw  (v2) edge[->] (v4);
\draw  (v3) edge[->] (v4);
\draw  (v5) edge[->] (v6);
\draw  (v7) edge[->] (v6);
\draw  (v8) edge[->] (v5);
\draw  (v8) edge[->] (v7);
\draw  (v9) edge[->] (v10);
\draw  (v10) edge[->] (v11);
\draw  (v9) edge[->] (v12);
\draw  (v12) edge[->] (v11);
\draw  (v13) edge[->] (v1);
\draw  (v13) edge[->] (v8);
\draw  (v13) edge[->] (v9);
\draw  (v13) edge[->] (v14);
\end{tikzpicture}
\caption{The first step of the construction of $ D $}\label{counterexample}
\end{figure} 

\subsection{Extending a flame by a $ \mathcal{P}\in \mathfrak{I}_D(v) $  keeping flame property}
\begin{ques}
Let $ D $ be a rooted digraph and let $ F\subseteq D $ be a vertex-flame. Is there necessarily  for every $ v\in V-r $ a $ 
\mathcal{P}\in 
\mathfrak{I}_D(v) $ such that $ F\cup \bigcup \mathcal{P} $ is a vertex-flame?
\end{ques}
 One can 
  prove that if the set of the new edges given by $ \mathcal{P} $ to $ F $ is $ \subseteq $-minimal among the choices $ 
  \mathfrak{I}_D(v) $ then 
  $ F\cup \bigcup \mathcal{P} $ is  a flame.  Calvillo-Vives proved (the edge version of) Theorem \ref{Flame 
 alaptetel} 
 based on 
this observation. Examples show that the premisses of the implication may fail to be satisfiable  in infinite digraphs.

\subsection{The edge version of flames and largeness}\label{edge version}
Let $ D $ be a rooted digraph. For $ v\in V-r $, we define $ \boldsymbol{\mathfrak{E}_{D}(v)} $ to be the set of those 
edge-disjoint  $ 
r\rightarrow v $ path-systems $ \mathcal{P} $ for which one can choose exactly one edge from each path in $ \mathcal{P} $ in 
such a way that the resulting $ C $ is an $ rv $-cut in $ D $. An $ L\subseteq D $ 
is $ \boldsymbol{D }$\textbf{-edge-large} if for every $ v\in V-r $ a $ \mathcal{P}\in \mathfrak{E}_{D}(v) $ lies in $ L $.
A rooted digraph $ D $ is an \textbf{edge-flame} if for all $ v\in V-r $ there is a system 
$ \mathcal{P} $ of edge-disjoint $ r\rightarrow v $ path such that $ A_{\text{last}}(\mathcal{P})=\mathsf{in}_F(v) $. 

\begin{ques}\label{question edge-flame}
Does there exist a $ D $-edge-large $ E $ which is an edge-flame for every  rooted digraph $ D $?
\end{ques}

It seems that most of the tools we developed works in the edge version as well. A major new difficulty is that a 
$ \mathcal{P}\in \mathfrak{E}_{D}(v) $ may give infinitely many ingoing edges to a vertex other than 
$ v $ 
(not just at 
most one as in the 
vertex version) and  therefore our quasi-flame approach is not sufficient itself to overcome this complication.

The edge version of the problem is stronger than the vertex version in the following sense. If the answer for  Question 
\ref{question edge-flame} is yes, 
then one can derive the vertex version from it  as we sketched in 
Remark \ref{vertex from edge}.
Similarly simple 
reduction in the other direction seems unlikely.  
The edge version analogues of all of our earlier open questions are also open.

\end{document}